\documentclass[a4paper]{article}

\usepackage[english]{babel}
\usepackage[cp1251]{inputenc}
\usepackage{amsmath}
\usepackage{amssymb}
\usepackage{amsthm}
\usepackage{amscd}
\usepackage{amsfonts}
\usepackage{newlfont}
\usepackage{enumerate}

\newtheorem{thm}{Theorem}[section]

\newtheorem{lem}[thm]{Lemma}

\newtheorem{definition}[thm]{Definition}

\newtheorem{corol}[thm]{Corollary}

\title{Properties of uniform tangent sets and \\ Lagrange multiplier rule \thanks{This work was partially supported by
the Sofia University "St. Kliment Ohridski" fund "Research \& Development"
under contract 80-10-220/22.04.2017 and by the Bulgarian National Scientific Fund under Grant DFNI-I02/10.}}

\author{Mira Bivas\footnotemark[2]\ \footnotemark[4]
\and Nadezhda Ribarska\footnotemark[3]\ \footnotemark[2]\ \footnotemark[5]
\and Mladen Valkov \footnotemark[3]\ \footnotemark[6]
}

\date{}

\begin{document}

\maketitle

\renewcommand{\thefootnote}{\fnsymbol{footnote}}

\footnotetext[2]{Institute of
Mathematics and Informatics, Bulgarian Academy of Sciences, G.Bonchev str., bl. 8, 1113 Sofia, Bulgaria}
\footnotetext[3]{Faculty of Mathematics and Informatics,  Sofia University, James Bourchier Boul. 5,  1126 Sofia, Bulgaria}
\renewcommand{\thefootnote}{\fnsymbol{footnote}}
\footnotetext[4]{email: mira.bivas@math.bas.bg}
\footnotetext[5]{email: ribarska@fmi.uni-sofia.bg}
\footnotetext[6]{email: mlado1992@abv.bg }

\begin{abstract}
The concept of uniform tangent sets was introduced and discussed in \cite{KR17}. This study is devoted to their further investigation and to generalization of the abstract Lagrange multiplier rule from \cite{KR17}.


\end{abstract}

\section{Introduction}

 Uniform tangent sets were introduced and briefly discussed in \cite{KR17}. This concept has proven to be useful for obtaining a nonseparation result, an abstract Lagrange multiplier rule, a necessary optimality condition of Pontryagin maximum principle type for optimal control problems in infinite-dimensional state space (c.f. again \cite{KR17}). The proposed approach reveals the importance of the uniformity of the approximation for obtaining necessary optimality conditions in infinite-dimensional setting.

 In the present paper we further study uniform tangent sets and enhance the above mentioned abstract Lagrange multiplier rule.

 In the first section the definitions of a uniform  tangent  set and a sequence uniform  tangent  set (taken from \cite{KR17}) are stated. It is proven that they are equivalent. Moreover, another equivalent characterization is obtained, in which further uniformity (with respect to the parameter $\lambda$) is assumed. As a corollary, the connection of uniform tangent sets with the classical concept of Clarke tangent cone is made clear. Another advantage of the main theorem in this section is that the work with sequence uniform tangent sets is significantly simplified. A sufficient condition for existence of a uniform tangent set generating the respective Clarke tangent cone is given.

 The second section contains a generalization and refinement of the Lagrange multiplier rule from \cite{KR17}, consisting in replacement of uniform tangent cones with uniform tangent sets and the allowance of non-strict minima.

\section{Uniform tangent sets}

Throughout the paper  $X$ is  a Banach
space, $ {\mbox{\bf B}} $ (resp. $\bar {\mbox{\bf B}}$) {  is } the open
(resp. closed) unit ball centered at the origin.

The following definitions are from \cite{KR17}:

\begin{definition} \label{UnCone}
Let   $S$ be a  closed subset of $X$ and $x_0$ belong to $S$. We
say that the bounded  set  $D_S(x_0)$ is a uniform  tangent  set to $S$ at
the point $x_0$ if for each $\varepsilon >0$ there exists
$\delta>0$ such that for each $v \in D_S(x_0)$
and for each point $x \in S \cap (x_0 + \delta \bar {\mbox{\bf
B}})$ one can find $\lambda>0$   for which      $S \cap (x+t (v + \varepsilon \bar {\mbox{\bf
B}}))$ is non empty  for each $t \in
[0,\lambda]$.
\end{definition}

\begin{definition} \label{UnConeSequences}
Let   $S$ be a closed subset of $X$ and $x_0$ belong to $S$. We
say that the bounded set  $D_S(x_0)$ is a sequence  uniform tangent
set to $S$ at the point $x_0$ if for each $\varepsilon >0$ there
exists $\delta>0$ such that for each $v \in D_S(x_0)$ and for each point $x \in S \cap (x_0 + \delta \bar
{\mbox{\bf B}})$ one can find a sequence of positive reals
$t_m\to 0$  for which $S
\cap (x+t_m (v + \varepsilon \bar {\mbox{\bf B}}))$ is non empty
 for each positive integer $m$.
\end{definition}

The next theorem is the main result in this section. The idea of its proof is the same as in the proof of Lemma 2.1 in \cite{Trei83}, namely the method of Bishop \& Phelps \cite{BP} as modified by Ekeland \cite{E}.

\begin{thm}\label{equiv_def}
Let   $S$ be a closed subset of $X$ and $x_0$ belong to $S$. The following are equivalent
\begin{enumerate}[(i)]
\item $D_S(x_0)$ is a uniform tangent set to $S$ at the point $x_0$
\item $D_S(x_0)$ is a sequence  uniform tangent set to $S$ at the point $x_0$
\item for each $\varepsilon >0$ there exist
$\delta>0$  and $\lambda>0$ such that for each $v \in D_S(x_0)$
and for each point $x \in S \cap (x_0 + \delta \bar {\mbox{\bf
B}})$ the set $S \cap (x+t (v + \varepsilon \bar {\mbox{\bf
B}}))$ is non empty  for each $t \in
[0,\lambda]$.
\end{enumerate}
\end{thm}

\begin{proof}
The implications $(iii) \implies (i) \implies (ii)$ follow from the definitions of a uniform tangent set and a  sequence  uniform tangent set. For the proof of the theorem to be complete, it suffices to prove that $(ii)$ implies $(iii)$.

Let $D_S(x_0)$ be a sequence  uniform tangent set to $S$ at the point $x_0$. Let us fix an $\varepsilon \in (0,1)$. Then, there exists $\tilde \delta>0$ such that for each $v \in D_S(x_0)$ and for each point $x \in S \cap (x_0 + \tilde \delta \bar{\mbox{\bf B}})$ one can find a sequence of positive reals
$t_m \to 0$  for which
\begin{equation} \label{suts}
S \cap (x+t_m (v + \varepsilon \bar {\mbox{\bf B}})) \neq \emptyset
\end{equation}
 for each positive integer $m$.

 Let $M$ be an upper bound of $\{ \|v\| \ | \ v \in D_S(x_0) \}$ and  let us set $\delta := \frac{\tilde \delta}{2}$ and $\lambda := \frac{\tilde \delta}{2(M+1)}$. Let us fix $\bar x \in S \cap (x_0 + \delta\bar{\mbox{\bf B}} )$ and $\bar v \in D_S(x_0)$. We are going to prove that $ S \cap (\bar x+t(\bar v+ \varepsilon \bar{\mbox{\bf B}} )) \neq  \emptyset$
  for each $t \in [0, \lambda]$.

  Let us assume the contrary, that is there exists $t_0 \in [0, \lambda]$ such that
 \begin{equation}\label{contr}
 S \cap (\bar x+t_0(\bar v+ \varepsilon \bar{\mbox{\bf B}} )) = \emptyset \, .
 \end{equation}

 Let us denote $A := S \cap (\bar x+[0, t_0](\bar v+ \varepsilon \bar{\mbox{\bf B}} ))$.
 Due to the choice of $\delta$ and $\lambda$, we have that
 \begin{align*}
[0, t_0](\bar v+ \varepsilon \bar{\mbox{\bf B}} ) \subset [0, \bar \lambda](M \bar{\mbox{\bf B}} + \varepsilon \bar{\mbox{\bf B}} ) \subset
 \frac{\tilde \delta}{2(M+1)}(M+\varepsilon)  \bar{\mbox{\bf B}}\subset \frac{\tilde \delta}{2}  \bar{\mbox{\bf B}}
 \end{align*}
 and therefore
  \begin{align}\label{S_delta}
  A = S \cap (\bar x+[0, t_0](\bar v+ \varepsilon \bar{\mbox{\bf B}} )) \subset S \cap \left(\bar x+ \frac{\tilde \delta}{2}  \bar{\mbox{\bf B}} \right) \subset S \cap \left(x_0+ \tilde \delta  \bar{\mbox{\bf B}} \right) \, .
  \end{align}

  We are going to define inductively a sequence $\{x_m\}_{m=1}^{\infty} \subset A$ and a sequence $\{\beta_m\}_{m=1}^{\infty} \subset (0,+\infty)$, satisfying
  \begin{equation}\label{xm}
  x_{m+1} \in S \cap \left(\bar x + \left(\sum_{i=1}^{m} \beta_i\right) (\bar v + \varepsilon \bar{\mbox{\bf B}}) \right) \mbox{ and } \sum_{i=1}^{m} \beta_i <t_0 \mbox{ for every } m \in \mathbb{N} \, .
  \end{equation}
 Let $x_1 := \bar x \in A$ and $\beta_0:=0$. Having obtained $x_m $ and $\beta_0, \dots, \beta_{m-1}$, let us set
  \begin{equation}\label{sup}
 \alpha_m := \sup\left\{\alpha \in  \left(0, t_0 - \sum_{i=1}^{m-1} \beta_i\right] \ | \ S \cap (x_m +\alpha(\bar v + \varepsilon \bar{\mbox{\bf B}})) \neq \emptyset  \right\}  \, .
  \end{equation}
 Since $x_m \in A$ by the inductive assumption, \eqref{S_delta} and \eqref{suts}  imply that the above set is nonempty and $\alpha_m > 0$.
 Therefore we can choose $$\beta_m \in \left(\max\{0, \alpha_m - 2^{-m}\}, \alpha_m\right]$$ such that $S \cap (x_m + \beta_m (\bar v + \varepsilon \bar{\mbox{\bf B}}) ) \neq \emptyset$ and then fix an arbitrary $x_{m+1}$ with
 $$x_{m+1} \in S \cap (x_m + \beta_m (\bar v + \varepsilon \bar{\mbox{\bf B}}) ) \ .$$
   We obtain that
 $$x_{m+1} \in S \cap \left(\bar x + \left(\sum_{i=1}^{m-1} \beta_i\right) (\bar v + \varepsilon \bar{\mbox{\bf B}}) + \beta_m (\bar v + \varepsilon \bar{\mbox{\bf B}}) \right) =$$ $$ =S \cap \left(\bar x + \left(\sum_{i=1}^{m} \beta_i \right)(\bar v + \varepsilon \bar{\mbox{\bf B}})  \right) $$
 due to the convexity of the ball. Moreover, $\sum_{i=1}^{m} \beta_i\le \sum_{i=1}^{m-1} \beta_i +\alpha_m\le t_0$ and $\sum_{i=1}^{m} \beta_i\not = t_0$ because $S \cap (\bar x+t_0(\bar v+ \varepsilon \bar{\mbox{\bf B}} )) = \emptyset$. Thus the construction of $\{x_m\}_{m=1}^{\infty}$ and $\{\beta_m\}_{m=1}^{\infty}$ is complete. 

 From  \eqref{xm} we have that
 $ \sum_{i=1}^{m} \beta_i \le t_0$ for every $ m \in \mathbb{N} $. Therefore the series $\sum_{i=1}^{\infty} \beta_i$ is convergent. 
 Hence, the sequence $\{x_m\}_{m=1}^{\infty}$ is Cauchy, because
 $$ x_{m+k} - x_{m} \in \sum_{i=m}^{m+k-1} \beta_i (\bar v + \varepsilon \bar{\mbox{\bf B}}) \subset \sum_{i=m}^{m+k-1} \beta_i (M  \bar{\mbox{\bf B}} + \varepsilon \bar{\mbox{\bf B}}) \subset \sum_{i=m}^{m+k-1} \beta_i (M+1) \bar{\mbox{\bf B}}$$
   implies
 $$  \|x_{m+k} - x_{m}\| < (M+1) \sum_{i=m}^{m+k} \beta_i   \, .$$
 Let us denote by $\hat x$ its limit point. The closedness of $S$ implies that $\hat x \in S$. Moreover, from \eqref{xm} we have that
  $$\hat x \in S \cap (\bar x + \beta (\bar v + \varepsilon \bar{\mbox{\bf B}}) )\subset A \, $$
  where $\beta :=  \sum_{i=1}^{\infty} \beta_i$. Indeed, the set $\bar x + \beta (\bar v + \varepsilon \bar{\mbox{\bf B}})$ is closed and
  $$\mathrm{dist}(x_m, \bar x + \beta (\bar v + \varepsilon \bar{\mbox{\bf B}}))\le \left(\sum_{i=m}^{\infty} \beta_i\right) (M+1)\longrightarrow_{m\to \infty} 0 \ .$$
  Due to $\beta\le t_0$ and \eqref{contr}, it follows that $\beta < t_0$.


Since
 $$x_{m+k} \in S \cap \left(x_m + \left(\sum_{i=m}^{m+k-1} \beta_i \right) (\bar v + \varepsilon \bar{\mbox{\bf B}})\right) \, , $$
the same reasoning as above implies
$$\hat x \in S \cap \left( x_m + \left(\sum_{i=m}^{\infty} \beta_i\right) (\bar v + \varepsilon \bar{\mbox{\bf B}}) \right) \ .$$
Therefore $\alpha_m - 2^{-m} <\beta_m < \sum_{i=m}^{\infty} \beta_i \le \alpha_m$ because of the definition of $\alpha_m$.
Let $t\in (0,t_0 - \beta)$ be arbitrary. If we assume that $S \cap (\hat x + t (\bar v + \varepsilon \bar{\mbox{\bf B}}) )\not = \emptyset$, then the inclusion above implies
$$
S \cap \left( x_m + \left(\sum_{i=m}^{\infty} \beta_i +t\right)(\bar v + \varepsilon \bar{\mbox{\bf B}}) \right)\not = \emptyset
$$
as well. Since
$$\sum_{i=m}^{\infty} \beta_i +t < \left( \beta -\sum_{i=1}^{m-1} \beta_i\right) +(t_0 -\beta)=t_0 - \sum_{i=1}^{m-1} \beta_i \ ,$$
from \eqref{sup} we have
$$\sum_{i=m}^{\infty} \beta_i +t \le \alpha_m < \beta_m +\frac{1}{2^m}$$
for every $m \in \mathbb{N}$. Therefore
$$0<t<\frac{1}{2^m} -\sum_{i=m+1}^{\infty} \beta_i\longrightarrow_{m\to \infty} 0 \ ,$$
a contradiction.
Thus we obtained that
\begin{align*}
S \cap (\hat x + (0, t_0 -\beta)(\bar v + \varepsilon \bar{\mbox{\bf B}})) =\emptyset \, ,
\end{align*}
which together with $\hat x \in A \subset S \cap \left(x_0+ \tilde \delta  \bar{\mbox{\bf B}} \right)$ contradicts \eqref{suts}. Hence, \eqref{contr} is not fulfilled which validates $(iii)$.
\end{proof}

In \cite{KR17} it is shown that the closure  of  a uniform tangent set  $D_S(x_0)$ to $S$ at $x_0$ is  a  uniform tangent set to $S$ at $x_0$ and the   convex hull   of a uniform tangent set is a uniform  tangent set. It is also obtained that
the   closure of  a sequence uniform tangent set is a sequence uniform tangent set. Now, we are able to show that
\begin{corol}\label{uts}
Let $S$ be a closed subset of  $X$ and let
$x_0\in S$. Let $D_S(x_0)$ be  a sequence  uniform tangent set to $S$ at the point $x_0 $. Then, the convex hull of $D_S(x_0)$ is a sequence  uniform  tangent set    to $S$ at $x_0$.
\end{corol}

\begin{proof}
The proof follows directly from Theorem \ref{equiv_def} and Lemma 2.5 in \cite{KR17}.
\end{proof}

We are going to use the original definition of the Clarke tangent cone in Banach spaces from \cite{Rock80}:
\begin{definition} \label{ClarkeCone}
Let   $S$ be a  closed subset of $X$ and $x_0$ belong to $S$. We
say that  $v\in X$  is a tangent vector to $S$ at
the point $x_0$ if for each $\varepsilon >0$ there exist
$\delta>0$ and $\lambda>0$ such that for each point $x \in S \cap (x_0 + \delta \bar {\mbox{\bf
B}})$  the set     $S \cap (x+t (v + \varepsilon \bar {\mbox{\bf
B}}))$ is non empty  for each $t \in
[0,\lambda]$.
The set of all tangent vectors to $S$ at $x_0$ is called Clarke tangent cone and is denoted by $\hat T_S(x_0)$.
\end{definition}

Theorem \ref{equiv_def} implies
\begin{corol}
Let $S$ be a closed subset of $X$ and let
$x_0\in S$. Let $D_S(x_0)$ be  an  uniform tangent set to $S$ at the point $x_0 $. Then, $D_S(x_0)$ is a subset of $\hat T_S(x_0)$, where $\hat T_S(x_0)$ is the Clarke tangent cone to $S$ at $x_0$.
\end{corol}

We obtain stronger connection between uniform tangent sets and the Clarke tangent cone, if the Clarke tangent cone is separable. To this end, we will need the definitions below:

\begin{definition}
The conical hull of a set $S$ is defined as
$$\mbox{cone } S := \mathbb{R}^{+}S = \{ \alpha s \ | \ \alpha > 0 \mbox{ and } s \in S \} \ .$$
\end{definition}

\begin{definition}
A set $S$ is said to generate the cone $C$, if $C$ is the closure of the conical hull of $S$. It is denoted by $\overline{\mbox{cone}} \, S \ .$
\end{definition}

Existence of a uniform tangent set generating the Clarke tangent cone is proven in the following lemma:

\begin{lem}
Let $X$ be a Banach space and let $S$ be its subset. Let $x\in S$ and let the Clarke tangent cone $\hat T_S(x)$ to $S$ at $x$ be separable.  Then there exists a uniform tangent set $D_S(x)$ to $S$ at $x$ which generates $\hat T_S(x)$.
\end{lem}

\begin{proof}
Let $\{ v_n\}_{n=1}^\infty$ be a dense subset of $\hat T_S(x)\cap \Sigma$ and let $\{ \lambda_n\}_{n=1}^\infty$ be an arbitrary sequence of positive numbers tending to zero. We claim that $D_S(x):=\{ \lambda_n v_n \ : \ n\in \mathbb{N} \}$ is a uniform tangent set to $S$ at $x$ which generates $\hat T_S(x)$. The last statement, that is $\overline{\textrm{cone}} \ D_S(x)=\hat T_S(x)$, is obvious. Now let $\varepsilon > 0$ be arbitrary. As $\| \lambda_n v_n\| =\lambda_n \longrightarrow_{n\to \infty} 0$, there exists $n_0\in \mathbb{N}$ such that $\| \lambda_n v_n\| <\varepsilon$ whenever $n > n_0$. Let $i\in \{ 1,2,\dots,n_0\}$. Since $\lambda_i v_i \in \hat T_S(x)$, there exists $\delta_i >0$ such that $\left( y+ t\left( \lambda_i v_i + \varepsilon \overline{B}\right) \right)\cap S\not = \emptyset$ whenever $y\in B_{\delta_i}(x)\cap S$ and $t\in (0,\delta_i )$. Put
$\delta := \min\left\{ \delta_1, \delta_2, \dots , \delta_{n_0}\right\} > 0$. Let $y\in B_{\delta}(x)\cap S$, $v\in D_S(x)$ and $t\in (0,\delta )$ be arbitrary. Then $v= \lambda_n v_n$ for some $n\in \mathbb{N}$. If $n\le n_0$, then $\delta \le \delta_n$ and thus $\left( y+ t\left( v + \varepsilon \overline{B}\right) \right)\cap S\not = \emptyset$. If $n > n_0$, then $\mathbf{0} \in v + \varepsilon \overline{B}$ and therefore $y\in \left( y+ t\left( v + \varepsilon \overline{B}\right) \right)\cap S\not = \emptyset$. This completes the proof.

\end{proof}

\section{Lagrange multiplier rule}

First, we need a technical result showing, roughly speaking, that by adding a strictly differentiable function to a given lower semicontinuous one, the uniform tangent set to the epigraph does not change significantly . Let us remind that if  $\varphi:X \longrightarrow \mathbb{R}  \cup \{+\infty\}$, its epigraph is the set
$epi\, \varphi :=\{ (x,p)\in X\times \mathbb{R}: \ \varphi(x)\le p\}$. We also remind the well-known definition of strictly differentiability (cf. \cite{Peano}):

\begin{definition}\label{strictlyDef}
A mapping $\psi: X\to \mathbb{R}$ is said to be strictly Fr\'echet
differentiable at $x_0 \in X$ if there exists a continuous
linear operator $\psi'(x_0): X \to \mathbb{R}$ such that for any
$\varepsilon > 0$ one can find $\delta>0$
  such that
  $$ |\psi(y) - \psi(x) - \psi'( x_0)(y - x)|\le\varepsilon  \|y - x\|$$
  whenever $\|x
- x_0\| < \delta$ and $\| y -x_0\| < \delta$.
\end{definition}

\begin{lem}\label{Mladen}
Let $\varphi:X \longrightarrow \mathbb{R}  \cup \{+\infty\}$ be lower semicontinuous and $x_0\in X$ be such that $\varphi(x_0)<+\infty$. Let $\psi: X\to \mathbb{R}$ be continuous on $X$ and strictly Fr\'echet
differentiable at $x_0$ and $\psi(x_0)=0$, $\psi'(x_0)=\mathbf{0}$. Let $D\subset X\times \mathbb{R}$. Then $D$ is a uniform tangent set to $epi \, \varphi$ at $(x_0,\varphi(x_0))$ if and only if  $D$ is a uniform tangent set to $epi (\varphi +\psi)$ at $(x_0,\varphi(x_0))$.
\end{lem}

\begin{proof}
Let $D\subset X\times \mathbb{R}$ be a uniform tangent set to $epi \, \varphi$ at $(x_0,\varphi(x_0))$. Then $D$ is bounded. Let $M>0$ be such that $M \geq \| (v,r)\| := max\{||v||,|r|\}$ for each $(v,r)\in D$.

Let us fix an arbitrary $\varepsilon>0$. Without loss of generality we may assume that $\varepsilon\le 1$.

The uniformity of the tangent set $D$ yields the existence of the positive reals $\tilde{\delta}>0$ and $\tilde{\lambda} >0$ such that whenever
$$(w,s) \in D, \ \ (x,p) \in (x_0+\tilde{\delta} \bar{B}, \ \  \varphi(x_0)+[-\tilde{\delta},\tilde{\delta}]) \cap epi \, \varphi \ \mbox{ and } \ t \in [0, \tilde{\lambda}]$$
 it is true that
$$\left[(x,p)+t\left((w,s)+\left(\frac{\varepsilon}{2} \bar{B}\right) \times \left[-\frac{\varepsilon}{2},\frac{\varepsilon}{2}\right]\right)\right] \cap epi \, \varphi\not = \emptyset \ .$$
Since $\varphi$ is lower semicontinuous at $x_0$ and $\varphi(x_0)<+\infty$, there exists $\delta_1>0$ such that $\varphi(x)>\varphi(x_0)-\tilde{\delta}$ for each $x\in B_{\delta_1}(x_0)$.
The function $\psi$ is continuous at $x_0$, therefore there exists $\delta_2>0$ such that $|\psi(x)-\psi(x_0)| = |\psi(x)| <\frac{\tilde{\delta}}{2}$ for each $x\in B_{\delta_2}(x_0)$. The strict  differentiability of $\psi$ at $x_0$ implies the existence of $\delta_3>0$ with
$$ |\psi(y) - \psi(x)| = |\psi(y) - \psi(x) - \psi'( x_0)(y - x)|\le\frac{\varepsilon}{2(M+1)}  \|y - x\|$$
  whenever $\|x - x_0\| < \delta_3$ and $\| y -x_0\| < \delta_3$ (we used that $\psi'(x_0)=\mathbf{0}$).

Now let us put $\delta := \min \left\{ \frac{\tilde{\delta}}{2},\delta_1,\delta_2, \frac{\delta_3}{2}\right\}$ and $\lambda:=\min \left\{ \frac{\delta_3}{3(M+1)}, \tilde{\lambda}\right\}$ and let us fix  arbitrary
$$(x,p)\in \left(\bar{B}_\delta(x_0)\times ((\varphi +\psi)(x_0)+[-\delta, \delta]\right)\cap epi (\varphi +\psi), \ (w,s) \in D  \mbox{ and }  t\in [0,\lambda ] .$$
Apparently $(\varphi +\psi)(x_0)=\varphi(x_0)$. Now $x\in \bar{B}_\delta(x_0)\subset \bar{B}_{\delta_1}(x_0)$ yields $\varphi(x)>\varphi(x_0)-\tilde{\delta}$.
On the other hand side, from  $x\in \bar{B}_\delta(x_0)\subset \bar{B}_{\delta_2}(x_0)$ we have $|\psi(x)| <\frac{\tilde{\delta}}{2}$ and, therefore, by the choice of $p$ and $\delta\le \frac{\tilde{\delta}}{2}$ we obtain
$$\varphi(x)=(\varphi +\psi)(x)-\psi(x)\le p +|\psi(x)|\le (\varphi +\psi)(x_0)+\delta +\frac{\tilde{\delta}}{2}\le \varphi(x_0)+\tilde{\delta} \ .$$
Thus we know that $\varphi(x)\in [\varphi(x_0)-\tilde{\delta},\varphi(x_0)-\tilde{\delta}]$. Apparently $(x,\varphi(x))$ belongs to the graph of $\varphi$, hence to its epigraph $epi \, \varphi$. This together with $\delta\le \tilde{\delta}$ yields $(x,\varphi(x)) \in (x_0+\tilde{\delta} \bar{B}, \ \  \varphi(x_0)+[-\tilde{\delta},\tilde{\delta}]) \cap epi \, \varphi$ and therefore
$$\left[(x,\varphi(x))+t\left((w,s)+\left(\frac{\varepsilon}{2} \bar{B}\right) \times \left[-\frac{\varepsilon}{2},\frac{\varepsilon}{2}\right]\right)\right] \cap epi \, \varphi\not = \emptyset \ ,$$
because $t\in [0,\lambda]\subset [0,\tilde{\lambda}]$. A straightforward computation shows that then there exists $w_x(t) \in w+\varepsilon \bar {B}$
such that
\begin{equation}\label{phi}
\frac{\varphi(x+t.w_x(t))-\varphi(x)}{t} \le s+\frac{\varepsilon}{2} \ .
\end{equation}
Now $\| w_x(t)\| \le \|w\| +\varepsilon \le M+1$ and so
$$\|( x+t.w_x(t))-x_0\| \le \| x-x_0\| + t \|w_x(t)\|\le \delta +\lambda (M+1)\le$$ $$\le  \frac{\delta_3}{2} +\frac{\delta_3}{3(M+1)}(M+1)<\delta_3 \ .$$
Having in mind that $\| x-x_0\|\le \delta <\delta_3$ as well, we obtain
$$|\psi(x+t.w_x(t)) - \psi(x)|  \le\frac{\varepsilon}{2(M+1)}  \| (x+t.w_x(t))- x\| =\frac{\varepsilon}{2(M+1)} t \|w_x(t)\|\le  \frac{\varepsilon}{2}t  \ .$$ Hence
\begin{equation}\label{psi}
\frac{\psi(x+t.w_x(t)) - \psi(x)}{t}\le \frac{\varepsilon}{2} \ .
\end{equation}
Adding up the inequalities (\ref{phi}) and (\ref{psi}), we obtain
$$\frac{(\varphi+\psi)(x+t.w_x(t))-(\varphi+\psi)(x)}{t} =$$ $$=\frac{\varphi(x+t.w_x(t))-\varphi(x)}{t} + \frac{\psi(x+t.w_x(t)) - \psi(x)}{t} \le s+\frac{\varepsilon}{2}+\frac{\varepsilon}{2}  = s+\varepsilon \ .$$
A straightforward computation again shows that the above inequality and the inclusion $w_x(t) \in w+\varepsilon \bar {B}$ imply
$$\left[(x,(\varphi+\psi)(x))+t\left((w,s)+\left({\varepsilon} \bar{B}\right) \times \left[-{\varepsilon},{\varepsilon}\right]\right)\right]\cap epi (\varphi+\psi)\not = \emptyset \ .$$
It remains to note that $(x,p)\in epi (\varphi+\psi)$ implies $p\ge (\varphi+\psi)(x)$ and thus
$$\left[(x,p)+t\left((w,s)+\left({\varepsilon} \bar{B}\right) \times \left[-{\varepsilon},{\varepsilon}\right]\right)\right]\cap epi (\varphi+\psi)\not = \emptyset $$
as well. This concludes the proof that $D$ is a uniform tangent set to $epi (\varphi+\psi)$. 

As $-\psi$ satisfies the same assumptions as $\psi$, the reverse implication ($D$ uniform tangent set to $epi (\varphi+\psi)$ implies $D$ uniform tangent set to $epi \, \varphi$) is exactly the same.
\end{proof}

The next definition is from \cite{KR17}. It uses a concept introduced in \cite{KRTs}, which
extents the notion of  a finite codimensional subset of $X$ (cf. Definition 1.5 in \cite{LY}).

\begin{definition}\label{solid}
Let $X$ be a Banach space and $S$ be a  subset of $X$. The set $S$
is said to be  quasisolid if its closed convex hull
$\overline{\mbox{\rm co}} \ S $ has nonempty interior in its
closed affine hull, i.e. if  there exists a point  $x_0  \in
\overline{\mbox{\rm co}}$ ${S}$ such that $\overline{\mbox{\rm
co}} \ \{S -y_0\} $ has nonempty interior in $\overline{\mbox{\rm
span}}$ $(S - y_0)$ (the closed subspace spanned by $ S-y_0 $).
\end{definition}

The following theorem is the main result in this section, namely an abstract Lagrange multiplier rule.

\begin{thm}\label{LagrangeM} Let us consider the optimization problem
$$\varphi (x)\to \min \ \mbox{ subject to } \ x \in S \ ,$$
where $\varphi:X \longrightarrow \mathbb{R}  \cup \{+\infty\}$ is lower semicontinuous and proper and $S$ is a closed subset of $X$.
Let $\bar x$ be a solution of the above problem. We set  $$\tilde S :=
S \times (-\infty, \varphi (\bar x)] \mbox{ and } \tilde  R
:= epi \, \varphi \subset X \times \mathbb{R}.$$
 Let $D_S$ be  a  uniform tangent set   to $S$ at  $ \bar x$,
 $D_{\tilde S}:= D_{  S} \times [-1, 0]$ and let $D_{\tilde { R}}$ be  a uniform tangent set    to  $\tilde {R} $ at  $(\bar x, \varphi (\bar x))$. We assume that the
set
$D_{\tilde S}-D_{\tilde R}$
is quasisolid. Then there exists a   pair
 $(\xi, \eta)\in
 X^*  \times  \mathbb{R}$ such that

\begin{enumerate}
\item [(i)] $(\xi,\eta)\not = (\mbox{\bf 0},0)$;
\item [(ii)] $\eta\in  \{0,1\}$;
 \item [(iii)] $\langle \xi ,v\rangle\ \le 0$  for every $v \in D_{S}$;

 \item [(iv)] $\langle \xi, w \rangle +\eta s\ge 0$   for every $(w,s)\in D_{\tilde R}$.

\end{enumerate}
\end{thm}

\begin{proof}
Let  $\psi (x):= \|x-\bar x\|^2$ for each
$x \in X$. Then $\psi$ is continuous on $X$ and strictly Fr\'echet differentiable at
$\bar x$ with zero derivative. Now Lemma \ref{Mladen} allows us to use Theorem 3.8 of \cite{KR17}.
\end{proof}


\begin{thebibliography}{99}

\bibitem{BP} E. Bishop, R. R. Phelps, The support functionals of a convex set, Proc.
Symp. Pure Math. 7 (1962), 27-35

\bibitem{E} I. Ekeland, On the variational principle,  J. Math. Anal. Appl., 47 (1974), 324–353

\bibitem{KR17} M. I. Krastanov and N. K. Ribarska, Nonseparation of Sets and Optimality Conditions, SIAM J. Control Optim., 55(3), 1598-1618

\bibitem{KRTs} M. I. Krastanov, N. K. Ribarska, Ts. Y. Tsachev, A Pontryagin maximum principle for
infinite-dimensional problems, SIAM Journal on Control and
Optimization, 49 (2011), No 5, 2155-2182

\bibitem{LY} X. J. Li, J.  Yong, Optimal control theory for infinite
dimensional systems,  Basel, Birkh\"auser (1994)

\bibitem{Peano} A. Nijenhuis, Strong derivatives and inverse mappings, Amer. Math.
Monthly 81  (1974), 969–-980

\bibitem{Rock80} R. T. Rockafellar, Generalized directional derivatives and subgradients of nonconvex functions, Canad. J. Math. 32(1980), 257-280

\bibitem{Trei83} Jay S. Treiman, Characterization of Clarke's tangent and normal cones in finite and infinite dimensions, Nonlinear Analysis: Theory, Methods \& Applications, Volume 7, Issue 7, July 1983, pp. 771-783


\end{thebibliography}
\end{document}